\documentclass{article}

\usepackage{epsfig}
\usepackage{graphicx}
\usepackage{amsbsy}
\usepackage{amsmath}
\usepackage{amsfonts}
\usepackage{amssymb}
\usepackage{textcomp}

\newcommand{\mcm}[3]{\newcommand{#1}[#2]{{\ensuremath{#3}}}} 

\mcm{\tuple}{1}{\langle #1 \rangle}
\mcm{\name}{1}{\ulcorner #1 \urcorner}
\mcm{\Fbb}{0}{\mathbb{F}}
\mcm{\Nbb}{0}{\mathbb{N}}
\mcm{\Zbb}{0}{\mathbb{Z}}
\mcm{\Rbb}{0}{\mathbb{R}}
\mcm{\Cbb}{0}{\mathbb{C}}
\mcm{\Qbb}{0}{\mathbb{Q}}
\mcm{\Bcal}{0}{\cal B}
\mcm{\Ccal}{0}{\cal C}
\mcm{\Dcal}{0}{\cal D}
\mcm{\Ecal}{0}{\cal E}
\mcm{\Fcal}{0}{\cal F}
\mcm{\Gcal}{0}{\cal G}
\mcm{\Hcal}{0}{\cal H}
\mcm{\Ical}{0}{\cal I}
\mcm{\Lcal}{0}{\cal L}
\mcm{\Mcal}{0}{\cal M}
\mcm{\Ncal}{0}{\cal N}
\mcm{\Pcal}{0}{{\cal P}}
\mcm{\Scal}{0}{{\cal S}}
\mcm{\Tcal}{0}{{\cal T}}
\mcm{\Ucal}{0}{{\cal U}}
\mcm{\Vcal}{0}{{\cal V}}
\mcm{\Wcal}{0}{{\cal W}}
\mcm{\Mfrak}{0}{\mathfrak M}
\newcommand{\easy}{\nobreak\hfill$\square$\medskip}

\mcm{\restric}{0}{\upharpoonright}
\mcm{\upset}{0}{\uparrow}
\mcm{\onto}{0}{\twoheadrightarrow}
\mcm{\smallNbb}{0}{{\small \mathbb{N}}}
\DeclareMathOperator{\preop}{op}
\mcm{\op}{0}{^{\preop}}

{\begin{array}{c}
\setlength{\unitlength}{1em}}%
{\end{array}}

\usepackage{amsthm}
\newtheorem{thm}{\sc{Theorem}}[section]

\newtheorem{rem}[thm]{\sc{Remark}}
\newtheorem{prop}[thm]{\sc{Proposition}}
\newtheorem{lem}[thm]{\sc{Lemma}}
\newtheorem{dfn}[thm]{\sc{Definition}}
\newtheorem{sublem}[thm]{\sc{Sublemma}}
\newtheorem{coroll}[thm]{\sc{Corollary}}

\newtheorem{question}[thm]{\sc{Open question}}
\newtheorem{eg}[thm]{\sc{Example}}
\newenvironment{example}{\begin{eg} \rm}{\end{eg}}

\newcommand{\itum}{\item[$\bullet$]}

\usepackage[all]{xy}
\DeclareMathOperator{\supp}{supp}
\DeclareMathOperator{\fin}{fin}
\DeclareMathOperator{\co}{co}
\DeclareMathOperator{\dual}{dual}
\DeclareMathOperator{\Sp}{cl}
\title{Thin sums matroids and duality}
\author{Hadi Afzali \\ \texttt{hadiafzali@math.uni-hamburg.de}  \and Nathan Bowler \\   \texttt{N.Bowler1729@gmail.com}}
\begin{document}

\maketitle

\begin{abstract}
Thin sums matroids were introduced to extend the notion of representability to non-finitary matroids.  We give a new criterion for testing when the thin sums construction gives a matroid. We show that thin sums matroids over thin families are precisely the duals of representable matroids (those arising from vector spaces). We also show that the class of tame thin sums matroids is closed under duality and under taking minors, by giving a new characterisation of the matroids in this class. Finally, we show that all the matroids naturally associated to an infinite graph are tame thin sums matroids.
\end{abstract}

\section{Introduction}

If we have a family of vectors in a vector space over some field $k$, we get a matroid structure on that family whose independent sets are given by the linearly independent subsets of the family. Matroids arising in this way are called {\em representable} matroids. Although many interesting finite matroids (eg. all graphic matroids) are representable, it is clear that any representable matroid is finitary and so many interesting examples of infinite matroids are not of this type. However, since the construction of many of these examples, including the algebraic cycle matroids of infinite graphs, is suggestively similar to that of representable matroids, the notion of {\em thin sums matroids} was introduced in \cite{RD:HB:graphmatroids}: it is a generalisation of representability which captures these infinite examples.

The basic idea is to take the vector space to be of the form $k^A$ for some set $A$, and to allow the linear combinations involved in the definition of dependence to have nonzero coefficients at infinitely many vectors, provided that they are well defined pointwise, in the sense that for each $a \in A$ there are only finitely many nonzero coefficients at vectors with nonzero component at $a$. Further details are given in Section \ref{pre}. Thin sums matroids need not be finitary.

There are some obvious questions about how well-behaved the objects given by this definition are. The first, and most obvious, question is whether the systems of independent sets defined like this are all really infinite matroids (in the sense of \cite{matroid_axioms}). Sadly, it is known that there are examples of set systems definable this way which are not matroids. Accordingly, we refer to such systems in general as {\em thin sums systems}, and only call them thin sums matroids if they really are matroids.

\begin{question}\label{whenmat}
Which thin sums systems are matroids?
\end{question}

A sufficient condition is given in \cite{RD:HB:graphmatroids}: a thin sums system over a family of vectors in $k^A$ is always a matroid when this family is {\em thin} - that is, when for each $a \in A$ there are only finitely many vectors in the family whose component at $a$ is nonzero. We show that, just as every representable matroid is finitary, so also every thin sums matroid over a thin family is cofinitary. Thus to get examples which are neither finitary nor cofinitary new ideas are needed. In fact we prove something stronger.

\begin{thm}
A matroid arises as a thin sums matroid over a thin family for the field $k$  iff it is the dual of a $k$-representable matroid.
\end{thm}

We will also provide a complete, if somewhat cumbersome, characterisation answering Question \ref{whenmat}. Although this characterisation allows us to simplify the proof of an old result of Higgs \cite{Higgs:axioms} characterising when the algebraic cycle system of a graph is a matroid, it is not completely satisfactory. The most nontrivial condition in the definition of matroids has not been removed. We will explore an analogy between thin sums systems and IE-operators which suggests that there is unlikely to be a simpler characterisation than the one we give.

The class of matroids representable over a field $k$ is very well behaved: It is closed under taking minors and (for finite matroids) under duality. This leads us naturally to ask the same questions about thin sums matroids.

\begin{question}
Is the class of thin sums matroids over $k$ closed under duality? Is it closed under taking minors?
\end{question}

The first part of this question is answered negatively in \cite{WILD}, where it is shown that there is a thin sums matroid whose dual is not a thin sums matroid. However, this counterexample is a very unusual matroid, in that it has a circuit and a cocircuit whose intersection is infinite. Such matroids are called {\em wild}, and matroids in which all circuit-cocircuit intersections are finite are called {\em tame}. Almost all standard examples of infinite matroids are tame - in fact, \cite{WILD} is the first paper to show that any matroid at all is wild. We are able to establish the following result:

\begin{thm}
The class of tame thin sums matroids over $k$ is closed under duality and under taking minors.
\end{thm}

We do this by giving an alternative characterisation of tame thin sums matroids for which this good behaviour is far more transparent.

Any finite graphic matroid is representable over every field. The situation for infinte graphs is a little more complex, in that there is more than one natural way to build a matroid from an infinite graph. In \cite{RD:HB:graphmatroids}, 6 matroids associated to a graph are defined in 3 dual pairs. We show that all 6 of these matroids are thin sums matroids over any field (this was already known for 1 of the 6, and one of the others was already known to be representable). 

In Section \ref{pre}, we will introduce some of the basic concepts, such as matroids, representability, and thin sums. We will also introduce the 6 graphic matroids mentioned above. Section \ref{corep} will be devoted to thin sums matroids on thin families, and their duality with representable matroids. In Section \ref{cond} we will develop our criterion for when a thin sums system is a matroid, and in Section \ref{galois} we will explain why we think there is unlikely to be a simpler characterisation. In Section \ref{duality} we will prove that the class of tame thin sums matroids is closed under duality and taking minors. Our account of why the various matroids associated to an infinite graph are thin sums matroids will be dispersed over all these sections: we give a summary of this aspect of the theory in Section \ref{summary}.

\section{ Preliminaries}\label{pre}

In this section, we introduce some terminology and concepts that we will use later. We also prove a few simple results about representability.

For any set $E$ let $\Pcal(E)$ be the power set of $E$. Recall that a matroid $M$ consists of a set $E$ (the ground set) and a set $\Ical \subseteq \Pcal(E)$
(the set of its independent sets), where $\Ical$ satisfies the following conditions: 

(I1) $\emptyset\in \Ical$.

(I2) $\Ical$ is closed under taking subsets.

(I3) For all $I\in \Ical\setminus \Ical^{max}$ and $I'\in \Ical^{max}$, there is an $x\in I' \setminus I$ such that $I+x\in \Ical$.

(IM) Whenever $I\subseteq X \subseteq E$ and $I\in \Ical$, the set $\{I' \in \Ical | I\subseteq I'\subseteq X\}$ has a maximal element.

Subsets of the ground set which are not independent are called {\em dependent}, and minimal dependent sets are called {\em  circuits} of the matroid. Maximal independent sets are called {\em bases}.
All the information about the matroid is contained in the set of its circuits, or of its bases. Further details can be found in \cite{matroid_axioms}, from which we take much of our notation.
If all the circuits of a matroid $M$ are finite then $M$ is called finitary. We use $M^*$ to denote the dual of $M$.

For any base $B$ and any $e \in E \setminus B$, there is a unique circuit $o_e$ with $e \in o_e \subseteq B + e$, called the {\em fundamental circuit} of $e$ with respect to $B$. Dually, since $E \setminus B$ is a base of $M^*$, for any $f \in B$ there is a unique cocircuit $b_f$ with $f \in b_f \subseteq E \setminus B + f$, called the {\em fundamental cocircuit} of $f$.

\begin{lem}\label{notone}
There is no matroid $M$ with a circuit $o$ and a cocircuit $b$ such that $|o \cap b| = 1$.
\end{lem}
\begin{proof}
Suppose for a contradiction that there were such an $M$, $o$ and $b$, with $o \cap b = \{e\}$. Let $B$ be a base of $M$ whose complement includes the coindependent set $b - e$. Let $I$ be a maximal independent set with $o - e \subseteq I \subseteq E \setminus b$ - this can't be a base of $M$ since its complement includes $b$, so by (I3), there is some $f \in B \setminus I$ such that $I + f$ is independent. Then by maximality of $I$, $f \in b$, and so $f = e$, so $o$ is independent. This is the desired contradiction.
\end{proof}

\begin{lem}\label{fdt}
 Let $M$ be a matroid and $B$ be a base.
Let $o_e$ and $b_f$ a fundamental circuit and a fundamental cocircuit with respect to $B$, then
\begin{enumerate}
 \item $o_e\cap b_f$ is empty or $o_e\cap b_f=\{e,f\}$ and
\item $f\in o_e$ iff $e\in b_f$.
\end{enumerate}
\end{lem}
\begin{proof}
(1) is immediate from Lemma \ref{notone} and the fact that $o_e \cap b_f \subseteq \{e, f\}$. (2) is a straightforward consequence of (1).
\end{proof}

\begin{lem}\label{o_cap_b}
 For any circuit $o$,  and any elements $e, f$ of $o$ there is a cocircuit $b$ such that $o\cap b=\{e,f\}$.
\end{lem}
\begin{proof}
Let $B$ be a base extending the independent set $o \setminus e$, so that $o$ is the fundamental circuit of $e$ with respect to $B$. Then the fundamental cocircuit of $f$ has the desired property.
\end{proof}

\begin{lem} \label{rest_cir}
 Let $M$ be a matroid with ground set $E = C \dot \cup X \dot \cup D$ and let $o'$ be a circuit of $M' = M / C \backslash D$.
Then there is an $M$-circuit $o$ with $o' \subseteq o \subseteq o' \cup C$.
\end{lem}
\begin{proof}
Let $B$ be any base of $M \restric_C$. Then $B \cup o'$ is $M$-dependent since $o'$ is $M'$-dependent.
On the other hand,  $B \cup o'-e$ is $M$-independent whenever $e\in o'$ since $o'-e$ is $M'$-independent.
Putting this together yields that $B \cup o'$ contains an $M$-circuit $o$, and this circuit must not avoid any $e\in o'$, as desired.
\end{proof}

\begin{coroll} \label{dualops}
Let $M$ be a matroid with ground set $E = C \dot \cup \{x\} \dot \cup D$. Then either there is a circuit $o$ of $M$ with $x \in o \subseteq C + x$ or there is a cocircuit $b$ of $M$ with $x \in b \subseteq D + x$, but not both. 
\end{coroll}
\begin{proof}
Note that $(M/D\backslash C)^* = M ^*/ C \backslash D$, and apply Lemmas \ref{notone} and \ref{rest_cir}. 
\end{proof}

In \cite{source_of_IE}, a {\em space} is defined to consist of a set $E$ together with an operator $\Pcal E \xrightarrow{S} \Pcal E$ such that $S$ preserves the order $\subseteq$ and satisfies $X \subseteq SX$ for any $X \subseteq E$. For example, for any matroid $M$ with ground set $E$ the associated {\em closure operator} $\Sp_M$, which sends $X$ to the set $$X \cup \{x \in E | (\exists o \in \Ccal(M)) x \in o \subseteq X + x\}$$ gives a space on the set $E$.

If $(E, S)$ is a space, the {\em dual space} is given by $(E, S^*$), where $S^*$ is the {\em dual operator} to $S$, sending $X$ to $X \cup \{x \in E | x \not \in S(E \setminus (X + x))$. Thus for sets $X$ and $Y$ with $X \dot \cup Y \dot \cup \{x\} = E$, we have
\begin{equation}\label{duop}
x \in SX \iff x \not \in S^*Y\, ,\tag{$\dagger$}
\end{equation}
and this completely determines $S^*$ in terms of $S$. Thus $S^{**} = S$. Also, by Corollary \ref{dualops}, for any matroid $M$ we have $\Sp_{M^*} = \Sp_M^*$.

For a space $(S, E)$, we say $S$ is {\em idempotent} if $S^2 = S$, and {\em exchange} if $S^*$ is idempotent. If $S$ is both idempotent and exchange, we call it an {idempotent-exchange} operator, or an {\em IE-operator} on $E$. Note that if $S$ is an IE-operator then so is $S^*$. For any matroid $M$ the operator $\Sp_M$ is an IE-operator. On the other hand, there are lots of IE-operators that don't come from matroids in this way. Some strong condition akin to IM is needed to pick out which IE-operators correspond to matroids.

We always use $k$ to denote an arbitrary field. The capital letter $V$ always stands for a vector space over $k$. For any set $A$, we write $k^A$ to denote the set of all functions from $A$ to $k$. For any function $E \xrightarrow{d} k$ the {\em support} $\supp(d)$ of $d$ is the set of all elements $e\in E$ such that $d(e)\neq 0$. {\em A linear dependence} of $E \xrightarrow{\phi} V$ is a map $E \xrightarrow{c} k$ such that 
$$\sum _{e\in E}c(e)\phi(e)=0$$
(here, as in the rest of this paper, we take this statement as including the claim that the sum is well-defined, i.e. that only finitely many summands are nonzero). For a subset $E'$ of $E$, we say such a $c$ is a {\em linear dependence} of $E'$ iff it is zero outside $E'$. Recall that a representable matroid is traditionally defined as follows.
\begin{dfn}
Let $V$ be a vector space. Then for any function $E \xrightarrow{\phi} V$ we get a matroid $M(\phi)$ on the ground set $E$, where we take a subset $E'$ of $E$ to be independent iff there is no nonzero linear dependence of $E'$. Such a matroid is called a representable or vector matroid.
\end{dfn}
Note that this is exactly the same as taking a family of vectors as the ground set and saying that a subfamily of this family is independent iff it is linearly independent. 

In \cite{RD:HB:graphmatroids}, there is an extension of these ideas to a slightly different context. Suppose now that we have a function $E \xrightarrow{f} k^A$. A {\em thin dependence} of $f$ is a map $E \xrightarrow{c} k$ such that, for each $a \in A$, 
$$\sum _{e\in E}c(e)f(e)(a)=0$$
This is not quite the same as a linear dependence (in $k^A$ considered as a vector space over $k$), since it is possible that the sum above might be well defined for each particular $a$ in $A$, but the sum
$$\sum_{e \in E} c(e)f(e)$$
might still not be well defined. To put it another way, there might be infinitely many $e \in E$ such that there is {\em some} $a \in A$ with $c(e)f(e)(a) \neq 0$, even if there are only finitely many such $e$ for each {\em particular} $a \in A$. We may also say $c$ is a thin dependence of a subset $E'$ of $E$ if it is zero outside of $E'$.

The word {\em thin} above originated in the notion of a {\em thin family} - this is an $f$ as above such that sums of the type given above are always defined; that is, for each $a$ in $A$, there are only finitely many $e \in E$ so that $f(e)(a) \neq 0$. Notice that, for any $E \xrightarrow{f} k^A$, and any thin dependence $c$ of $f$, the restriction of $f$ to the support of $c$ is thin.

Now we may define thin sums systems. 

\begin{dfn} 
Consider a family $E \xrightarrow{f} k^A$ of functions and declare a subset of $E$ as independent iff there is no nonzero thin dependence of that subset. Let $M_{ts}(f)$ be the set system with ground set $E$ and the set of all independent sets given in this way. We call $M_{ts}(f)$ the {\em thin sums system} corresponding to $f$. Whenever $M_{ts}(f)$ is a matroid it is called a thin sums matroid.
\end{dfn}

Note that every  dependent set in a representable matroid or thin sums system induces a linear or thin dependence and vice versa; therefore, we normally talk about such dependences instead of dependent sets.

Not every thin sums system is a matroid%
\footnote{See Section \ref{cond} for a couple of examples.} but it is known that if $f$ is thin then $M_{ts}(f)$ always is a matroid. The existing proof for this is technical and we shall not review it here. However, this fact will follow from the results in Section \ref{corep}. Next we explore the connection between representable and thin sums matroids. Recall that for any infinite matroid $M$ all its finite circuits are circuits of a matroid%
\footnote {This is easy to prove. See \cite{union2}.}, called the {\em finitarisation} of $M$.  

\begin{prop}\label{finthinrep}
For any thin sums matroid $M_{ts}(f)$, the finitarisation of $M_{ts}(f)$ is a representable matroid.
\end{prop}
\begin{proof}
For any family $E \xrightarrow{f} k^A$ of functions, a thin dependence of $f$ with finite support is also a linear dependence of $f$ as a family of vectors, and conversely any linear dependence of $f$ as a family of vectors is a thin dependence of $f$. 
\end{proof}

Now let's try to answer to the question: Which matroids arising from graphs are representable or thin sums matroids? It is easy to see that any algebraic cycle matroid is a thin sums matroid (in fact, this was one motivation for the definition of thin sums matroids). Recall that for any graph $G$ which does not contain a subdivision of the Bean graph,
$$
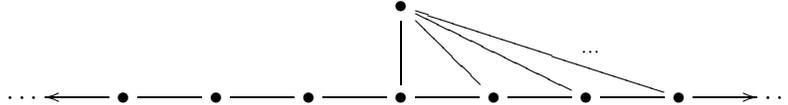
\xymatrix{&&&& \bullet \ar@{-}[d] \ar@{-}[dr] \ar@{-}[drr] \ar@{-}[drrr] \ar@{}[drrrr]|{\cdots} &&&&\\ \cdots & \ar[l] \bullet \ar@{-}[r] & \bullet \ar@{-}[r] & \bullet \ar@{-}[r] & \bullet \ar@{-}[r] & \bullet \ar@{-}[r] & \bullet \ar@{-}[r] & \bullet \ar[r]& \cdots \\}$$
the edge sets of cycles and double rays of $G$ are circuits of a matroid%
\footnote{This has been proved in \cite{Higgs:axioms}. However, later we will be able to give a simpler proof of this.} $M_A(G)$ on the edge set of $G$ which is called the {\em algebraic cycle matroid} of $G$. In fact, even when $G$ does contain a subdivision of the Bean graph we shall still denote this system of sets by $M_A(G)$, and call it the {\em algebraic cycle system} of $G$.

\begin{prop}\label{algthin}
For any graph $G$ the algebraic cycle system of $G$ is a thin sums system over every field.
\end{prop}
\begin{proof}
First we give an arbitrary orientation to every edge of $G$, making $G$ a digraph. For any edge $e$ of $G$ define a function $V(G) \xrightarrow{f(e)} k$ where for any $v\in V(G)$ $f(e)(v)$ is $1$ if $e$ originates from $v$, $-1$ if it terminates in $v$, and $0$ if they don't meet each other. We show that $D$ is dependent in $M_A(G)$ iff it is dependent in $M_{ts}(f)$. If $D$ is dependent in $M_A(G)$, then it contains a cycle or a double ray. Let $D'\subseteq D$ be the edge set of this cycle or double ray. Give a direction to $D'$. For any edge $e\in D$, define $c(e)$ to be $1$ if $e$ is an edge of $D'$ and they have the same directions, $-1$ if $e$ is in $D'$ and they have different directions, and $0$ if they don't meet. Now clearly we have $\sum_{e \in D'} c(e)f(e)(v)=0$ for any vertex $v$ of $G$, so $c$ is a thin dependence of $D$. Conversely if $D$ is dependent in $M_{ts}(f)$, then whenever a vertex $v$ is an end of an edge in $D$, it has to be the end of at least two edges in $D$. Now it is not difficult to see that $D$ has to contain a cycle or a double ray.
\end{proof}

Recall that the edge sets of finite cycles give the circuits of a matroid $M_{FC}(G)$, the {\em finite cycle matroid} of $G$. An argument almost identical to the one above shows that this matroid is always representable. Dually, the edge sets of finite bonds give the circuits of a matroid $M_{FB}(G)$, the {\em finite bond matroid} of $G$%
\footnote{See \cite{RD:HB:graphmatroids} for a description of the various cycle and bond matroids which may be associated to a graph.}. Similar ideas allow us to show that for any graph $G$, $M_{FB}(G)$ is also representable. 

\begin{prop}\label{fbond}
For any graph $G$ $M_{FB}(G)$ is representable over every field $k$.
\end{prop}
\begin{proof}
We start by giving fixed directions to every edge, cycle and finite bond. Let $O$ be the set of all cycles of $G$ and for any edge $e \in E(G)$ define a function $O \xrightarrow {\phi(e)}k$ such that for any 
$o\in O$, $\phi(e)(o)$ is $1$ if $e\in o$ and they have the same directions, $-1$ if $e\in o$ and they have different directions, and $0$ if $e$ isn't an edge of $o$. This defines a map $E(G) \xrightarrow{\phi} k^O$. We will show $M(\phi)=M_{FB}(G)$.

We need to show that $D \subseteq E(G)$ is dependent in $M_{FB}(G)$ iff the it is dependent in $M(\phi)$. If $D$ is dependent in $M_{FB}(G)$ then it contains a finite bond $D'$. For any edge $e\in D'$ define $c(e)$ to be $1$ if $D'$ and $e$ have the same directions, and $-1$ if they have different directions, and $0$ if they don't meet.
Now consider a fixed cycle $o$ which meets $D'$. Clearly $D'$ has two sides and this cycle has to traverse $D'$ from the first side to the second side as many times as
 it traverses $D'$ from the second side to the first. As a result, for any $o\in O$ we have $\sum_{e \in E} c(e) \phi(e)(o)=0$ and so $c$ is a linear dependence of $D$.
 
 Conversely, suppose that $D$ is dependent in $M(\phi)$, and let $D'$ be the support of any thin dependence of $D$.
 Whenever the edge set of a cycle meets $D'$, they have to meet in at least two edges, which means $D'$ (and so also $D$) meets every spanning tree and so it contains a bond and so it is a dependent set in $M_{FB}(G)$.
 \end{proof}

 Recall that for any graph $G$ the (possibly infinite) bonds of $G$ are the circuits of a matroid $M_B(G)$ on the edge set of $G$%
\footnote{See \cite{RD:HB:graphmatroids}.}. In the above proof, we could exchange the role of finite bonds and arbitrary bonds and see that $M_B(G)$ is a thin sums matroid. We could also exchange the role of finite cycles and arbitrary bonds, and finite bonds and finite cycles, to get another proof of the fact that $M_{FC}(G)$ is representable. 
 It is not difficult to see that the finite cycle matroid and the bond matroid of a graph $G$ are dual to each other%
\footnote{See \cite{RD:HB:graphmatroids}.}.
  
 As has been shown in \cite{RD:HB:graphmatroids}, for any graph $G$ the circuits of the dual of the finite bond matroid of $G$ are given by the topological circles in a topological space associated to $G$. For this reason, $M^*_{FB}(G)$ is called the topological cycle matroid of $G$, and denoted $M_C(G)$.  In the next section, we shall show that $M_C(G)$ is also a thin sums matroid. 

We will only give a brief summary of the construction of the topological space behind the topological cycle matroid. A ray is a one-way infinite path. 
Two rays are edge-equivalent if for any finite set $F$ of edges there is a connected component of $G\setminus F$ that contains subrays of both rays. The equivalence classes of 
this relation are the edge-ends of $G$; we denote the set of these edge-ends by $\Ecal(G)$. Let us view the edges of $G$ as disjoint topological copies of [0,1], 
and let $X_G$ be the quotient space obtained by identifying these copies at their common vertices. The set of inner points of an edge $e$ will be 
denoted by $e\!\!^{^o}$. We now define a topological space $||G||$ on the point set of $X_G\cup \Ecal(G)$ by taking as our open sets the union of 
sets $\tilde{C}$, where $C$ is a connected component of $X_G\setminus Z$ for some finite set $Z\subset X_G$ of inner points of edges, and $\tilde{C}$ is obtained 
from $C$ by adding all the edge-ends represented by a ray in $C$. For any $X\subseteq ||G||$ we call $\{e \in E(G) | e\!\!^{^o} \subseteq X\}$ the edge set of $X$.
A subspace $C$ of $||G||$ that is homeomorphic to $S^1$ is a {\em topological circle} in $||G||$. In \cite{RD:HB:graphmatroids} is shown that the edge sets of these circles in $||G||$ are the circuits of $M^*_{FB}(G)$. 

\section{Representable matroids and thin sums}\label{corep}
In this section we elucidate the connections between representable matroids and thin sums matroids. First we show that any representable matroid is a thin sums matroid, so
thin sums matroids are a generalisation of representable matroids. After that we will characterise the dual of an arbitrary representable matroid and show that not only is every representable matroid a thin sums matroid but every matroid whose dual is representable is also a thin sums matroid. In fact, our last result even is stronger; we show that the duals of representable matroids are precisely the thin sums matroids for thin families. Since the finite bond matroid of any graph is representable, this implies in particular that its dual, the topological cycle matroid, is a thin sums matroid.

As usual, let $V^*$ be the dual of the vector space $V$ (that is, the vector space consisting of all linear maps from $V$ to $k$).

\begin{thm}\label{repisthin}
Consider a map $E \xrightarrow{\phi} V$ and the representable matroid $M(\phi)$. For any $e\in E$ and $\alpha \in V^*$ define $E \xrightarrow{f} k^{V^*}$ by $f(e)(\alpha):=\alpha.\phi(e)$. 
Then,
$$M(\phi)=M_{ts}(f).$$
In particular, $M(\phi)$ is a thin sums matroid.
\end{thm}
\begin{proof}
We show that $I$ is independent in $M_{ts}(f)$ iff $I$ is independent in $M(\phi)$. Suppose that $I$ is independent in $M_{ts}(f)$. Suppose that 
$E \xrightarrow{c} k$ is any linear dependence of $\phi$ that is $0$ outside $I$. For any $\alpha \in V^*$ we have,
$$\sum_{e\in E} c(e).f(e)(\alpha)=\sum_{e\in E} c(e)\alpha.\phi(e)=\alpha\left(\sum_{e\in E} c(e)\phi(e)\right) =0.$$
As $\alpha$ was arbitrary we have $\sum_{e\in E} c(e)f(e)=0$ and since $I$ is independent in $M_{ts}(f)$ $c$ must be the $0$ map. So $I$ is also independent in $M(\phi)$.

Conversely, suppose that $I$ is independent in $M(\phi)$. Suppose $E \xrightarrow{c} k$ is any thin dependence of $f$ that is $0$ outside $I$.
Let $I'=\supp(c)$. Since $I' \subseteq I$, $I'$ is also independent in $M(\phi)$, so (by extending the image of $I'$ by $\phi$ to a basis of $V$) we can define a linear map 
$V\xrightarrow{\alpha_{I'}} k$ such that for any $i\in I'$, $\alpha_{I'}(\phi(i))=1$. As the restriction of $f$ to $I'=\supp(c)$ is thin and for any $i\in I'$ $f(i)(\alpha_{I'})=\alpha_{I'}(\phi(i))=1$, $I'$ has to be finite. So for every $\alpha \in V^*$,
$$\alpha\left(\sum_{e\in E} c(e)\phi(e)\right)=\sum_{e\in E} c(e)\alpha.\phi(e)=\sum_{e\in E} c(e)f(e)(\alpha)=0.$$
Since this is true for every $\alpha \in V^*$, we get that $\sum_{e\in I'} c(e)\phi(e)=0$ which means $c$ must be a linear dependence and so must be 0. Therefore $I$ is also indpendent in $M_{ts}(f)$.
\end{proof}

Now let's see how we can move from a representable matroid to its dual. Let's start with a family $E \xrightarrow{\phi} V$. let $C_{\phi}$ be the set of all linear dependences of $\phi$. For any $e\in E$ and 
$c\in C_{\phi}$ define $E \xrightarrow{\widehat{\phi}} k^{C_{\phi}}$ by $\widehat{\phi}(e)(c) := c(e)$. Clearly $\widehat{\phi}$ is a thin family of functions. On the other hand,    
if we let $D_f$ be the set of thin dependences of a thin family $E \xrightarrow{f} k^A$, we get a map $E \xrightarrow{\overline{f}} k^{D_{f}}$ where for $e\in E$ and $d\in D_{f}$ $\overline{f}(e)(d):= d(e)$. 
These processes are, in a sense, inverse to each other.

\begin{lem}\label{backforth}
For any thin family $E \xrightarrow{f} k^A$, a map $d: E \to k$ is a thin dependence of $f$ iff it is a thin dependence of $\widehat{\overline{f}}$.
\end{lem}
\begin{proof}
First, suppose that $d$ is a thin dependence of $f$. Then for any $c \in C_{\overline{f}}$ we have 
 \[
  \sum_{e\in E} d(e)\widehat{\overline{f}}(e)(c)=\sum_{e\in E} d(e)c(e)=\sum_{e\in E} c(e)\overline{f}(e)(d)=0,
   \]
so $d$ is also a thin dependence of $\widehat{\overline{f}}$.

Now suppose that $d$ is a thin dependence of $\widehat{\overline{f}}$. For any $a \in A$, let $E \xrightarrow{c_a} k$ be defined by the equation $c_a(e) = f(e)(a)$. Since $f$ is thin, $c_a(e)$ is nonzero for only finitely many values of $e$. Since also for any thin dependence $d'$ of $f$ we have
 \[
  \sum_{e\in E} c_a(e)\overline{f}(e)(d')= \sum_{e\in E} c_a(e )d'(e)=\sum_{e\in E} d'(e )f(e)(a)=0,
   \]    
and so $c_a \in C_{\overline{f}}$. Now, since $d$ is a thin dependence of $\widehat{\overline{f}}$, we have
 \[
  \sum_{e\in E} d(e)f(e)(a) = \sum_{e\in E} d(e)c_a(e) = \sum_{e\in E}d(e) \widehat{\overline{f}}(e)(c_a)=0.
   \]    
Since $a$ was arbitrary, this says exactly that $d$ is a thin dependence of $f$.
\end{proof}

An analogous argument shows that for any map $\phi: E \to V$, the linear dependences of $\overline{\widehat{\phi}}$ are exactly those of $\phi$. We can also show that these inverse processes correspond to duality of matroids.

\begin{thm}\label{dualise}
For any map $\phi: E \to V$ we have,
 \[
 M^*(\phi)=M_{ts}(\widehat{\phi}).
  \]    
\end{thm}
\begin{proof}
Suppose we have a set $E_1$ which is dependent in the dual of $M(\phi)$: that is, it meets every basis of $M(\phi)$. Let $E_2 = E \setminus E_1$, so $E_2$ doesn't contain any basis of $E$ - that is, $E_2$ doesn't span this matroid, and we can pick $e_1 \in E_1$ such that $\phi(e_1)$ isn't in the span of the family $(\phi(e) | e \in E_2)$. Consider a basis $B_2$ for this span, and extend $B_2+  \phi (e_1)$ to a basis $B$ for V, and define a map $B \xrightarrow{h_0} F$ such that $h_0(\phi (e_1)):=1$,  and  otherwise 0. Finally, extend $h_0$ to a linear map $V \xrightarrow{h} F$. Now, for any linear dependence $c$ of $ \phi$ we have
 \[
	\sum_{e \in E} (h\cdot\phi)(e) \widehat{\phi}(e)(c)=h\left(\sum_{e\in E} c(e)\phi(e)\right)=0
\]  
So $h \cdot \phi$ is a thin dependence of $\widehat{\phi}$, and since it is 0 outside $E_1$, $E_1$ is dependent with respect to $\widehat{\phi}$.

Conversely, suppose that $E_1$ is dependent in $M_{ts}(\widehat{\phi})$, so that there is a nonzero thin dependence $d$ of $\widehat{\phi}$ which is 0 outside $E_1$. We want to show that $E_1$ meets every basis of $M(\phi)$, so suppose for a contradiction that there is such a basis $B$ which it doesn't meet. Pick $e_1 \in E_1$ so that $d$ is nonzero at $e_1$. We can express $\phi(e_1)$ as a linear combination of vectors from the family $(\phi(e) | e \in B)$ - that is, there is a linear dependence $c$ of $\phi$ which is nonzero only on $B$ and at $e_1$, with $c(e_1) = 1$. But then 
\[
 d(e_1) =  \sum_{e\in E} d(e)c(e) =\sum_{e\in E} d(e)\widehat{\phi}(e)(c) = 0,
\]   
which is the desired contradiction. Thus $E_1$ does meet every basis of $M(\phi)$, so it is dependent in the dual of $M(\phi)$.
\end{proof}

\begin{coroll}
For any thin family $ E \xrightarrow{f} k^A$ we have,
\[
 M_{ts}(f)=M^*(\overline{f}).
  \]    
In particular $M_{ts}(f)$ is a cofinitary matroid. 
\end{coroll}
\begin{proof}
This is immediate from theorem \ref{dualise}, since by lemma \ref{backforth} we have $M_{ts}(f) = M_{ts}(\widehat{\overline{f}})$.
\end{proof}

\section{A sufficient condition for $M_{ts}$ to be a matroid}\label{cond}

Throughout this section, $f$ will denote a map $E \xrightarrow{f} k^A$ for some sets $A$ and $E$ and field $k$. Since so many examples of matroids are of the form $M_{ts}(f)$ for some such $f$, it would be good to be able to characterise when the set system $M_{ts}(f)$ is a matroid. Although this set system clearly satisfies the axioms (I1) and (I2), it need not satisfy either (I3) or (IM). Thus the algebraic cycle system of the Bean graph satisfies (IM) but not (I3) (as we shall soon show).

On the other hand, we can also define a thin sums system which fails to satisfy (IM). Let $E = \Nbb \times \{0, 1\}$, and define a function $E \xrightarrow{f} \Qbb^{\Nbb}$ by $f((n,0))(i) = i^n$ and $f((n, 1))(i) = -1$ if $n = i$ and 0 otherwise. Thus for any thin dependence $c$ of $f$, there can only be finitely many $n \in \Nbb$ with $c((n, 0))$ nonzero, and the remaining values of $c$ are determined by the polynomial expression
\begin{equation}\label{polyeq}
c((i, 1)) = \sum_{n \in \Nbb}c((n, 0)) i^n \, .
\end{equation}
In particular, if $c$ is 0 outside of $\Nbb \times \{0\}$, then this polynomial must be the 0 polynomial and so $c$ must be the 0 function. This shows that $\Nbb \times \{0\}$ is thinly independent for $f$. To show that $M_{ts}(f)$ doesn't satisfy (IM), we shall show that there is no maximal thinly independent superset of $\Nbb \times \{0\}$. More precisely, we shall show that, for a subset $X$ of $\Nbb$, the set $\Nbb \times \{0\} \cup X \times \{1\}$ is thinly independent if and only if $\Nbb \setminus X$ is infinite. In fact, the same argument as that above shows that this set is thinly independent whenever $\Nbb \setminus X$ is infinite, since the only polynomial which is zero in infinitely many places is the zero polynomial. Conversely, if $\Nbb \setminus X$ is finite, then pick some nonzero polynomial $\sum_{n = 0}^N a_n x^n$ with roots at all elements of $\Nbb \setminus X$, and define $c((n, 0))$ to be $a_n$ for $n \leq N$ and 0 otherwise. Define $c((i, 1))$ by the polynomial formula \eqref{polyeq}. Then $c$ is a nontrivial thin dependence which is 0 outside $\Nbb \times \{0\} \cup X \times\{1\}$, so that set is thinly dependent.

We shall argue in the next section that some condition like (IM) is unavoidable, but we can at least get rid of the condition (I3). We do this by defining for each $f$ a different set system $M^{\co}_{ts}(f)$, which satisfies (I3) in addition to (I1) and (I2), and such that if it satisfies (IM) then $M_{ts}(f)$ is a matroid (in fact, in such cases $M^{\co}_{ts}(f) = M^*_{ts}(f)$). 

We will make use of a compactness lemma, corresponding to the compactness of a topological space which (so far as we know) has not been introduced in the literature. We therefore introduce it here.

\begin{dfn}
An {\em affine equation} over a set $I$ with coefficients in $k$ consists of a sequence $(\lambda_i \in k| i \in I)$ such that only finitely many of the $\lambda_i$ are nonzero and an element $\kappa$ of $k$. 

A sequence $(x_i | i \in I)$ is a {\em solution} of the equation $(\lambda, \kappa)$ iff $\sum_{i \in I} \lambda_i x_i = \kappa$ (we shall also slightly abuse notation by using expressions like this as names for equations). x is a solution of a set $E$ of equations iff it is a solution of every equation in $E$.
\end{dfn}

The proof of the following Lemma is based on an argument of Bruhn and Georgakopoulos
\cite{basis}.

\begin{lem}\label{compsolve}
If every finite subset of a set $E$ of affine equations over $I$ with coefficients in $k$ has a solution then so does $E$.
\end{lem}
\begin{proof}
Let \Ical\ be the set of all subsets $I'$ of $I$ such that every finite subset of $E \cup \{x_i = 1 | i \in I'\}$ has a solution. \Ical\ is nonempty since it contains $\emptyset$, and the union of any chain of elements of \Ical\ is an upper bound for those elements in \Ical. So by Zorn's lemma, \Ical\ has a maximal element $I_m$. Let $E_m = E \cup \{x_i = 1 | i \in I_m\}$. 

\begin{sublem}
For any $i \in I$ there is a finite set $K_i \subseteq E_m$ and an element $u_i \in k$ such that every solution $x$ of $E_i$ has $x_i = u_i$.
\end{sublem}
\begin{proof}
Suppose for a contradiction that we can find $i_0 \in I$ where this fails - it follows that $i_0 \not \in I_m$. We shall show that $I_m \cup \{i_0\} \in \Ical$, contradicting the maximality of $I_m$. Let $K$ be any finite subset of $E_m$. By assumption, $K$ has two solutions $s$ and $t$ with $s_{i_0} \neq t_{i_0}$. Since $k$ is a field, we can find $A$ and $B$ in $k$ such that $A + B = 1$ and $As_{i_0} + Bt_{i_0} = 1$. Then for each equation $\sum_{i \in I} \lambda_i x_i = \kappa$ in $K$ we have $\sum_{i \in I}\lambda_i(As_i + Bt_i) = A\sum_{i \in I} \lambda_is_i + B \sum_{i \in I} \lambda_it_i = A\kappa + B\kappa = \kappa$, so $As + Bt$ is a solution of $K$, and since $As_{i_0} + Bt_{i_0} = 1$ it is even a solution of $K \cup \{x_{i_0} = 1\}$, as required.
\end{proof}

It is now enough to show that the $u_i$ introduced above form a solution of $E$. Let $e$, given by $\sum_{i \in I} \lambda_i x_i = \kappa$, be any equation in $E$, and let $S$ be the finite subset of $I$ on which the $\lambda_i$ are nonzero. Let $K = \bigcup_{i \in S} K_i \cup \{e\}$. Since $K$ is finite, it has a solution $x$. For each $i \in S$ we have $x_i = u_i$, so $\sum_{i \in I} \lambda_iu_i = \sum_{i \in I} \lambda_ix_i = \kappa$, so the $u_i$ form a solution of $e$ and, since $e$ was arbitrary, they form a solution of $E$.
\end{proof}

This lemma is all we will really need. However, it looks like it ought to correspond to some sort of compactness, and indeed it does. 

\begin{dfn}
For any affine equation $e$ over $I$ with coefficients in $k$, let $C_e$ be the set of solutions of $e$. For any finite set $E$ of affine equations, let $C_E = \bigcup_{e \in E} C_e$.
The {\em affine Zariski topology} on $k^I$ is that with the $C_E$ as its basic closed sets.
\end{dfn}

The reason for this name is the analogy between this definition and the Zariski topology on $k[X]$ for a finite set $X$.

\begin{thm}
The affine Zariski topology is compact.
\end{thm}
\begin{proof}
Let $\Ecal$ be a set of finite sets of affine equations, such that for any finite subset $K$ of $\Ecal$ the set $\bigcap_{E \in K} C_E$ is nonempty. What we need to show is that $\bigcap_{E \in \Ecal} C_E$ is also nonempty. Let $X$ be $\prod_{E \in \Ecal} E$, with the product topology. For each finite $K \subseteq \Ecal$, let $X_K$ be the subset of $X$ consisting of all $(e_E | E \in \Ecal)$ such that $\{e_E | E \in K\}$ has a solution. $X_K$ is closed and nonempty since $K$ is finite. For any finite family $(K_j | j \in J)$ of such $K$ we have that $\bigcap_{j \in J} X_{K_j} \supseteq X_{\bigcup_{j \in J} K_j}$, so it is nonempty. Since $X$ is compact, the intersection of all the $X_K$ is also nonempty, so we can pick an element $e$. Then we know that every finite subset of $\{e_E | E \in \Ecal\}$ has a solution, so by Lemma \ref{compsolve} there is a solution $x$ of the whole set of equations. But then $x$ lies in $\bigcap_{E \in \Ecal} C_E$, which is therefore nonempty.
\end{proof}

\begin{lem}\label{mindep}
Let $d$ be a thin dependence of $f$. Then $\supp(d)$ is a union of minimal dependent sets of $M_{ts}(f)$.
\end{lem}
\begin{proof}
Let $I = \supp(d)$. It suffices to show that for any $e_0 \in I$ there is a minimal dependent set which contains $e_0$ and is a subset of $I$. We begin by fixing such an $e_0$.

For any $a\in A$ there are only finitely many $e\in I$ with $f(e)(a)\neq 0$, so for any $a\in A$ we get an affine equation $\sum_{e\in I} f(e)(a) x_e=0$ over $I$. Let $\Ecal$ be the set of all affine equations arising in this way.
Let $\Scal$ be the set of all subsets $I'$ of $I$ such that every finite subset of $\Ecal \cup \{x_e=0|e\in I'\} \cup \{x_{e_0}=1\}$ has a solution. Since $d \restric_I$ is a solution of all equations in $\Ecal$, $(d\restric_I)/d(e_0)$ is a solution of all equations in 
$\Ecal \cup\{x_{e_0}=1\}$, so $\emptyset \in \Ecal$. $\Scal$ is also closed under unions of chains, so by Zorn's lemma it has in has a maximal element $E_m$. Now by Lemma \ref{compsolve} there is some solution $d'$ of all the equations in
 $\Ecal \bigcup \{x_e=0|e\in E'\} \bigcup \{x_{e_0}=1\}$. Since $d'$ solves all the equations in $\Ecal$, its extension to $E$ taking the value 0 outside $I$ is a thin dependence of $f$.

We shall show that $D:=\supp(d')=E\setminus E_m$ is the desired minimal dependent set. If it were not, there would have to be a nonzero thin dependence $d''$ with $\supp(d'') = \supp(d') - e_0$. But then for any $e_1 \in \supp(d'')$, we have that $d' - \frac{d'(e_1)}{d''(e_1)}d''\restric_I$ is a solution of $x_{e_1} = 0$ in addition to the equations solved by $d'$, which contradicts the maximality of $E_m$.
\end{proof}

\begin{coroll}\label{spans}
If $M_{ts}(f)$ is a matroid, and $E' \subseteq E$, then $e \not \in E'$ is in the closure of $E'$ iff there is a thin dependence $d$ with $\supp(d) \subseteq E' \cup \{e\}$ and $d(e) = 1$.
\end{coroll}
\begin{proof}
If there is such a $d$, by Lemma \ref{mindep} we can find a minimal dependent set $D$ with $e\in D\subseteq \supp(d)$. As $D\setminus  \{e\}\subseteq E'$ is independent, $e \in \Sp(E')$. If $e \in \Sp(E')$ then there is a circuit $D$ with 
$e\in D\subseteq E' \cup \{e\}$. Let $d$ be a thin dependence with $\supp(d) = D$. Then $d(e) \neq 0$ since $D - e_0$ is indpendent: scaling if necessary, we can take $d(e)=1$.
\end{proof}

\begin{coroll}\label{coind}
Let $M_{ts}(f)$ be a matroid. Then a subset $I$ is independent in $M^*_{ts}(f)$ iff for every $i \in I$ there is a thin dependence $d_i$ of $f$ such that $d_i(i) = 1$ and $d_i$ is $0$ on the rest of $I$.
\end{coroll}
\begin{proof}
We recall that $I$ is independent in $M^*_{ts}(f)$ iff $\Sp(I^c)=E$. Now apply Lemma \ref{spans}.
\end{proof}

This motivates the definition we promised at the start of this section, of the set system $M_{ts}^{\co}$. 

\begin{dfn}
A subset $I$ of $E$ is {\em coindependent} iff for every $i \in I$ there is a thin dependence $d_i$ of $f$ such that $d_i(i) = 1$ and $d_i$ is 0 on the rest of $I$. The set system $M_{ts}^{\co}(f)$ has ground set $E$ and consists of the coindependent subsets of $E$. 
\end{dfn}

Thus by Corollary \ref{coind}, when $M_{ts}(f)$ is a matroid, $M_{ts}^{\co}(f)=M^*_{ts}(f)$.

\begin{lem}\label{coindplus}
Let $I$ be coindependent and $i_0\not\in I$. If there is a thin dependence $d$ which is nonzero at $i_0$ and $0$ on $I$, then $I + i_0$ is coindependent.
\end{lem}
\begin{proof}
Suppose that $(d_i |i\in I)$ witnesses the coindependence of $I$. Let $d'_{i_0}=d/d(i_0)$, and for $i\in I$ let $d'_i=d_i-d_i(i_0)d'_{i_0}$. Then $(d'_i | i\in I + i_0)$ witnesses the coindependence of $I + i_0$.
\end{proof}

We can now show that $M_{ts}^{co}(f)$ is always dual, in a sense, to $M_{ts}(f)$.

\begin{lem}\label{bases}
Let $I \subseteq E$. $I$ is a maximal independent set with respect to $f$ iff $E \setminus I$ is a maximal coindependent set with respect to $f$. 
\end{lem}
\begin{proof}
Suppose first of all that $I$ is a maximal independent set, and let $i \in E \setminus I$. Let $d_i$ witnesses the dependence of $I \cup \{i\}$. We must have $d_i(i) \neq 0$, so without loss of generality $d_i(i) = 1$. But then the $d_i$ witness the coindependence of $E \setminus I$. We can't have $(E \setminus I)  + i$ coindependent for any $i \in I$, since the corresponding $d_i$ would witness dependence of $I$.

So suppose instead for a contradiction that $E \setminus I$ is a maximal coindependent set but $I$ is dependent, as witnessed by some thin dependence $d$ of $I$. There must be $i_0 \in I$ with $d(i_0) \neq 0$ so, by Lemma \ref{coindplus}, $(E \setminus I) + i_0$ is coindependent, contradicting the maximality of $E \setminus I$. Thus $I$ is independent. For each $i \in E \setminus I$, $I \cup \{i\}$ is dependent, as witnessed by $d_i$, and so $I$ is also maximal.
\end{proof}

$M_{ts}^{\co}(f)$ evidently satisfies (I1) and (I2). 

\begin{lem}\label{(I3)}
$M_C^{\co}$ satisfies (I3).
\end{lem} 
\begin{proof}
Suppose we have a maximal coindependent set $J$, and a nonmaximal coindependent set $I$. We have to show that we may extend $I$ with a point from $J$. Since $I$ is nonmaximal, we can choose $i_0 \not \in I$ with $I + i_0$ still coindependent. Since by Lemma \ref{bases} $E \setminus J$ is  independent, there is $i_1 \in J$ with $d_{i_0}(i_1) \neq 0$. Then by Lemma \ref{coindplus} $I  + i_1$ is coindependent.
\end{proof}

We can now give our slightly simplified criterion for when a thin sums system is a matroid.

\begin{thm}\label{comat}
If $M_{ts}^{\co}(f)$ satisfies (IM), then
$$(M_{ts}^{co}(f))^*=M_{ts}(f).$$
In particular, $M_{ts}(f)$ is a matroid.
\end{thm} 
\begin{proof}
$M_{ts}^{co}(f)$ evidently satisfies (I1) and (I2), and satisfies (I3) by Lemma \ref{(I3)}, so it is a matroid. It is clear from Lemma \ref{bases}, that every independent set of $(M_{ts}^{co}(f))^*$ is also independent $M_{ts}(f)$. Conversely, let $I$ be an independent set of $M_{ts}(f)$. Then let $J$ be a maximal independent set of $M_{ts}^{\co}(f)$ not meeting $I$. It suffices to show that $J$ is a base of $M_{ts}^{\co}(f)$. Suppose not, for a contradiction: then there is some $i \in I$ with $J + i$ coindependent. But then since $I$ is independent, the corresponding $d_i$ is nonzero at some $j \not \in I$, and by Lemma \ref{coindplus} we deduce that $J + j$ is coindependent,  contradicting the maximality of $J$.
\end{proof}

We now return to the question of when the algebraic cycle system $M_A(G)$ of a graph $G$ is a matroid. It evidently satisfies (I1) and (I2). A little trickery shows that $M_A(G)$ has a maximal independent set $B$. First, we pick a maximal collection $A$ of disjoint rays in $G$, then we can take $B$ to be any maximal set of edges including all the rays in $A$ but not including any cycle and not connecting any 2 of the rays in $A$ (both these steps are possible by Zorn's Lemma). $B$ can't include a double ray, by maximality of $A$. A slight refinement of this argument shows that $M_A(G)$ always satisfies (IM). So we just need to determine whether $M_A(G)$ satisfies (I3). 

In fact, as we mentioned in Section \ref{pre}, it was shown by Higgs in \cite{Higgs:axioms} that $M_A(G)$ is a matroid iff $G$ doesn't contain any subdivision of the Bean graph:
$$
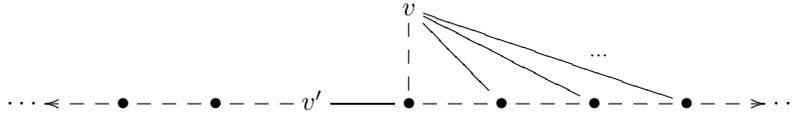
\xymatrix{&&&& v \ar@{--}[d] \ar@{-}[dr] \ar@{-}[drr] \ar@{-}[drrr] \ar@{}[drrrr]|{\cdots} &&&&\\ \cdots & \ar@{-->}[l] \bullet \ar@{--}[r] & \bullet \ar@{--}[r] & v' \ar@{-}[r] & \bullet \ar@{--}[r] & \bullet \ar@{--}[r] & \bullet \ar@{--}[r] & \bullet \ar@{-->}[r]& \cdots \\}$$
The algebraic cycle system of this graph doesn't satisfy (I3) - the dashed edges above form a maximal independent set, but there is no way to extend the nonmaximal independent set consisting of the edges meeting $v$ and those to the left of $v'$ by an edge from this set. It is, however, not at all easy to see that if $G$ doesn't contain a subdivision of the Bean graph then $M_A(G)$ satisfies (I3). In fact, Higgs didn't follow this route - the interested reader can check that his claim (3) (which is the combinatorial heart of the paper) is exactly the criterion obtained from Theorem \ref{comat} in this case. We are now in a position to give a more direct argument.

\begin{thm}[Higgs]\label{higgs}
Suppose that $G$ contains no subdivision of the Bean graph. Then $M_A(G)$ is a matroid.
\end{thm}
\begin{proof}
We say a cut $b$ of $G$ is a {\em nibble} if one side (called the {\em small side}: the other side is the {\em large side}) of $b$ is connected and contains no rays. Suppose, for a contradiction, that there are a nibble $b$ and an algebraic cycle $a$ meeting $b$ infinitely often. Then $a$ must be a double ray. Let $T$ be a spanning tree of the small side of $b$. We can pick any vertex $v_0$ in this tree to serve as its root, and consider the subtree $T'$ consisting of the paths from $v_0$ to $a$ in $T$. Since $T'$ is rayless and has infinitely many leaves there must (by K\"{o}nig's Lemma) be a vertex $v$ in this tree of infinite degree. The paths from $v$ to $a$ in $T'$, together with $a$ itself, now contain a subdivision of the Bean graph, contrary to our supposition. So we can conclude that a nibble and an algebraic cycle can only meet finitely often.

In fact we can say more, using the ideas of Section \ref{pre}. Pick directions for every edge, algebraic cycle and nibble of $G$. Let $A$ be the set of all agebraic cycles of $G$, and for any edge $e \in E(G)$ define a function $A \xrightarrow {f(e)}k$ such that for any 
$a\in A$, $f(e)(a)$ is $1$ if $e\in a$ and they have the same directions, $-1$ if $e\in a$ and they have different directions, and $0$ if $e$ isn't an edge of $a$. This gives a map $E(G) \xrightarrow{f} k^A$. We shall show that $M_{ts}^{\co}(f)=M_A(G)$.

First, we show that any coindependent set $I$ for $f$ is $M_A(G)$-independent. Suppose for a contradiction that $I$ includes an algebraic cycle $a$, and pick any $i \in a$. Then $\sum_{e \in E}d_i(e)f(e)(a) = f(i)(a) \neq 0$, which is the desired contradiction.

For any nibble $b$ of $G$, define the map $E(G) \xrightarrow{d_b} k$ such that $d_b(e)$ is $1$ if $e\in b$ and they have the same directions, $-1$ if $e\in b$ and they have different directions, and $0$ if $e$ isn't an edge of $b$. For any algebraic cycle $a$, $a$ must traverse $b$ the same number of times in each direction (if it is a double ray, the rays in both directions must eventually end up in the large side of $b$). Traversals one way contribute a $+1$ term to $\sum_{e \in E} d_b(e)f(e)(a)$, and traversals the other way contribute a $-1$ term, so this sum is always 0. That is, each $d_b$ is a thin dependence of $f$. 

Now if a set $I$ isn't coindependent then there is some $i \in I$ such that no thin dependence is nonzero at $i$ and 0 on the rest of $I$. In particular, considering the thin dependences $d_b$ above, there is no nibble $b$ with $b \cap I = \{i\}$. Thus if the connected components of $I - i$ containing the endpoints of $i$ are distinct then each contains a ray, so $I$ contains a double ray. Otherwise, both ends of $i$ are in the same component, so $I$ contains a cycle. In either case, $I$ contains an algebraic cycle. 

We have shown that the $M_A(G)$-independent sets are exactly the coindependent sets, so they satisfy (I3) by Lemma \ref{(I3)}. We have already checked the remaining axioms.
\end{proof}

\begin{rem}
This argument also shows a little more - namely that the dual of $M_A(G)$ is the thin sums matroid $M_{ts}(f)$. We have shown that every nibble is thinly dependent. On the other hand, if a set $I$ contains no nibble, so that every connected component of the complement of $I$ contains a ray, then for each $i$ in $I$ there is an algebraic cycle meeting $I$ only in $i$, so $I$ is thinly independent. Thus the cycles of the dual of $M_A(G)$ are exactly the minimal nibbles. This is also shown in \cite{RD:HB:graphmatroids}, where minimal nibbles are called skew cuts.
\end{rem}

\section{Galois Connections}\label{galois}

In this section, we will present a new perspective on the definition of thin sums set systems, which we believe makes it unlikely that any criterion much simpler than (IM) will allow us to distinguish which such systems are matroids. To do this, we shall show that thin sums systems are determined by closed classes for a particular Galois connection. We shall note that each IE-operator gives a closed class for a very similar Galois connection. Since in that case (IM) seems to be necessary to pick out the class of matroids, we think something similar will be needed for thin sums systems also. This section is not essential for what follows, and may be skipped.

Since Galois connections are not widely known, we shall review here the small portion of the theory that we shall require.

\begin{dfn}
Let $A$ be a set, and $R$ a symmetric relation on $A$. The {\em Galois connection} induced by $R$ is the function $(\Pcal A \xrightarrow{p} \Pcal A)$ given by $p(A') = \{a \in A | (\forall a' \in A') aRa'\}$.
\end{dfn}

For the remainder of this section we shall always take $A$, $R$ and $p$ to refer in this way to the constituents of a general Galois connection.

\begin{example}\label{orth}
Let $V$ be a vector space with an inner product $\langle -,- \rangle$. We say 2 vectors $v$ and $w$ are {\em orthogonal} iff $\langle v,w \rangle = 0$. This gives a relation from $V$ to itself, and so induces a Galois connection as above. $p$ is given by the function $\Pcal V \to \Pcal V$ that sends a subset of $V$ to its {\em orthogonal complement}, which is always a subspace of $V$.
\end{example}

\begin{lem}\ 
\begin{itemize}
\itum For $A' \subseteq A'' \subseteq A$, $p(A'') \subseteq p(A')$
\itum For $A' \subseteq A$, $A' \subseteq p^2(A')$
\end{itemize}
\end{lem}
\begin{proof}
To prove the first property, note that for any $a \in p(A'')$, for any $a' \in A' \subseteq A''$ we have $aRa'$, so that $a \in p(A')$. To prove the second property, note that for any $a' \in A'$, for any $a \in p(A')$ we have $a'Ra$, so that $a' \in p^2(A')$.
\end{proof}

\begin{lem}\label{galprops}
For $A' \subseteq A$, the following are equivalent:
\begin{itemize}
\itum $A' = p^2(A')$.
\itum $A'$ is in the image of $p$.
\end{itemize} 
\end{lem}
\begin{proof}
The first statement clearly implies the second. Suppose the second is true, and let $A' = p(A'')$. Then we have $A'' \subseteq p^2(A'')$, so $p(A'') \supseteq p^3(A'')$, that is $A' \supseteq p^2(A')$. Since we also know $A' \subseteq p^2(A')$, we have $A' = p^2(A')$ as required.
\end{proof}

In such cases, we say $A'$ is a {\em closed} subset of $A$ (with respect to this Galois connection). It is immediate from Lemma \ref{galprops} that $p$ restricts to an order reversing idempotent automorphism of the poset of closed subsets of $A$. For any closed set $A'$, $p(A')$ is called the {\em dual} closed set to $A'$.

\begin{example}
If, in Example \ref{orth}, $V$ is finite dimensional, then the closed sets for this Galois connection are precisely the subspaces of $V$.
\end{example}

\begin{example}\label{graphcon}
Let $G$ be a finite graph, and let $V$ be the free vector space $\Fbb_2^E$ over $\Fbb_2$ on the set of $E$ edges of $G$. We can identify subsets of $E$ with vectors in $V$: each subset gets identified with its characteristic function. There is a standard inner product on this space, with $\langle v, w\rangle = \sum_{e \in E} v(e)w(e)$. Then the cuts of $G$ form a subspace of $V$, which is the orthogonal complement of the subspace of $V$ generated by the cycles of $G$.  Thus the cycles and the bonds of $G$ generate dual closed classes in the associated Galois connection.
\end{example}

Let $E$ be any set, and define a relation $R_1$ from $\Pcal E$ to itself by letting $X R_1 Y$ iff $|X \cap Y| \neq 1$. This slightly odd relation is motivated by the fact that it holds between any circuit and any cocircuit in a matroid. We shall show that each matroid with ground set $E$ induces a closed subset of $\Pcal E$ in the associated Galois connection, and that the dual matroid induces the dual closed subset. In fact, we can go further and get such a result for idempotent-exchange operators (see Section \ref{pre} for a definition of this concept).

\begin{dfn}
Let $S$ be an IE-operator on a set $S$. A set $X \subseteq E$ is {\em S-closed} iff $SX = X$. A subset $X$ is an $S$-scrawl iff for each $x \in X$ it is true that $x \in S(X - x)$. The set of $S$-scrawls is denoted $\Scal(S)$. 
\end{dfn}

Thus if $M$ is a matroid then a set is $\Sp_M$-closed iff it is $M$-closed and is a $\Sp_M$-scrawl iff it is a union of $M$-circuits.

\begin{lem}\label{scrawls}
Let $S$ be an IE-operator on $E$, and let $X \subseteq E$. Then $X$ is $S$-closed iff $E \setminus X$ is an $S^*$-scrawl. 
\end{lem}
\begin{proof}
Note that by \eqref{duop} of Section \ref{pre} for any $x \in E \setminus X$ we have $x \not \in SX$ iff  $x \in S^*(E \setminus X - x)$.
\end{proof}

\begin{coroll}\label{matscrawls}
Let $M$ be a matroid with ground set $E$, and let $s \subseteq E$ be a set which never meets an $M$-cocircuit in just one point. Then $s$ is a union of $M$-circuits. \easy
\end{coroll}

\begin{thm}
Let $S$ be an IE-operator on a set $E$, and let $p$ be given as above by the Galois connection associated to $R_1$. Then $\Scal(S) = p(\Scal(S^*))$.
\end{thm}
\begin{proof}
We must show that a subset $X$ of $E$ is in $\Scal(S)$ iff it is in $p(\Scal(S^*))$. 

First of all, suppose that $X \in \Scal(S)$, and pick any $X' \in \Scal(S^*)$. Suppose for a contradiction that $|X \cap X'| = 1$, and call the unique element of this set $x$. Then $x \in S(X - x)$ and so $x \in S(E \setminus X')$, which contradicts the fact that by Lemma \ref{scrawls}, $E \setminus X'$ is $S$-closed. Since $X'$ was arbitrary we get that $X \in p(\Scal (S^*))$.

Now suppose instead that $X \in p(\Scal(S^*))$. For any $x \in X$, $S(X - x)$ is $S$-closed, since $S$ is idempotent, so $E \setminus S(X - x) \in \Scal(S^*)$ by Lemma \ref{scrawls}. So $X \cap (E \setminus S(X - x))$, which is a subset of $\{x\}$, can't have just one element. So $x \not \in E \setminus S(X - x)$ and so $x \in S(X - x)$. Since $x$ was arbitrary, $X \in \Scal(S)$.
\end{proof}

Thus although every matroid corresponds to a closed class for such a Galois connection, not every such closed class corresponds to a matroid: the far more general collection of IE-operators gives rise to many such closed classes which don't come from matroids. Thus, in order to determine which closed classes for these Galois connections correspond to matroids, some condition akin to (IM) is essential.

However, there is a similar Galois connection whose closed classes capture the information behind thin sums systems. Let $E$ be a set, and $k$ a field.   We have a relation $R_2$ from $k^E$ to itself with $c R_2 d$ iff
$$\sum_{e \in E}c(e) d(e) = 0 \, .$$
Here, as usual, we take this to include the statement that the sum is well defined, i.e. that only finitely many of the summands are nonzero.

Just as in Example \ref{orth}, any closed set is necessarily a subspace of the vector space $k^E$. The link between this relation and the relation $R_1$ defined above is that, since no sum evaluating to zero can have precisely one nonzero term, if $c R_2 d$ then there can't be just 1 $e \in E$ at which both are nonzero. Explicitly, $c R_2 d \Rightarrow \supp(c)R_1 \supp(d)$.

From any closed class, we can define a corresponding set system.

\begin{dfn}
For any closed set $C$ with respect to $R_2$, we say a subset $I$ of $E$ is {\em $C$-independent} iff the only $c \in C$ which is zero outside $I$ is the 0 function. Otherwise, $I$ is {\em $C$-dependent}. The {\em thin sums system} $M_C$ corresponding to $C$ is the system of $C$-independent subsets of the ground set $E$.
\end{dfn}

We shall now show that this notion corresponds to the usual notion of a thin sums system.

\begin{prop}\label{olddef}
Suppose we have a function $E \xrightarrow{f} k^A$. Let $D$ be the set of functions $d_a \colon e \mapsto f(e)(a)$ with $a \in A$. Then $M_{ts}(f) = M_{p(D)}$.
\end{prop}
\begin{proof}
It is enough to show that the elements of $p(D)$ are exactly the thin dependences for $f$. But using the substitution given above, the condition that $c \in p(D)$ namely that for each $a \in A$ 
$$\sum_{e \in E}c(e)d_a(e)\, ,$$
becomes the condition that for each $a \in A$
$$\sum_{e \in E}c(e)f(e)(a)$$
which is the condition for $c$ to be a thin dependence for $f$.
\end{proof}

Because thin sums systems correspond in this way to closed classes for the Galois connection correponding to $R_2$, and a condition like (IM) seems necessary to pick out the matroids amongst the closed classes for $R_1$, it is likely that some condition akin to (IM) will also be needed to distinguish which thin sums systems are matroids. On the other hand, the evident similarity of this connection to the sort employed in example \ref{graphcon} provides another indication of why the various types of cycle and bond matroids corresponding to a graph are all thin sums systems.

\section{Tameness and duality}\label{duality}

One very natural question about the class of thin sums matroids is whether or not it is closed under matroid duality: the fact that the class of representable matroids was not closed under duality was a key motivation for introducing extensions of this class, such as the class of thin sums matroids. Sadly, the class of thin sums matroids is not closed under duality: a counterexample is given in \cite{WILD}. However, that counterexample involves a matroid with a very unusual property: it has a circuit and a cocircuit whose intersection is infinite. Matroids with this property are called {\em wild} matroids, and those in which every circuit-cocircuit intersection is finite are called {\em tame}. The main result of this section will be that the class of tame thin sums matroids is closed under duality.

This class includes all the interesting examples arising from graphs: any finitary or cofinitary matroid must be tame, and this includes the finite and topological cycle matroids as well as the bond and finite bond matroids of a given graph. We showed in the proof of Theorem \ref{higgs} that the algebraic cycle and skew cuts matroids are also tame. 

A natural strategy for showing that the dual of a thin sums matroid is again a thin sums matroid is suggested by the results of Section \ref{corep}. These results suggest that in attempting to construct the representation $E \xrightarrow{\overline{f}}k^{\overline{A}}$ of $M^*_{ts}(f)$ we should take $\overline{A}$ to be the set of all thin dependencies of $f$, and define $\overline{f}(e)(c)$ to be $c(e)$. However, this natural attack fails to work, even if $M_{ts}(f)$ is tame, as our next example shows.

\begin{example}
Let $G$ be the graph 
$$
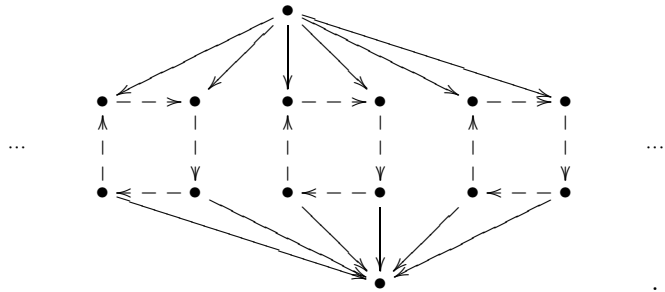
\xymatrix{&&& \bullet \ar[dll] \ar[dl] \ar[d] \ar[dr] \ar[drr] \ar[drrr] &&&&\\ \ar@{}[d]|\cdots & \bullet \ar@{<--}[d] \ar@{-->}[r] & \bullet \ar@{-->}[d] & \bullet \ar@{<--}[d] \ar@{-->}[r] & \bullet \ar@{-->}[d] & \bullet \ar@{<--}[d] \ar@{-->}[r] & \bullet \ar@{-->}[d]  &  \ar@{}[d]|\cdots \\& \bullet \ar[drrr] \ar@{<--}[r] & \bullet \ar[drr] & \bullet \ar[dr] \ar@{<--}[r] & \bullet \ar[d] & \bullet \ar[dl] \ar@{<--}[r] & \bullet \ar[dll] &\\&&&& \bullet &&& .\\}$$
We may represent the algebraic cycle matroid of $G$ as $M_{ts}(f)$ as in the proof of Proposition \ref{algthin}. Recall that for any edge $e$ of $G$ the function $V(G) \xrightarrow{f(e)} k$ is given by taking $f(e)(v)$ to be $1$ if $e$ originates from $v$, $-1$ if it terminates in $v$, and $0$ if they don't meet each other. Thus the function which takes the value $1$ on the dotted edges and 0 elsewhere is a thin dependence of $f$. So no function with support given by the skew cut consisting of the vertical dotted edges can be a thin dependence of $\overline{f}$ as given above. That is, for this matroid and this definition of $\overline{f}$, we have $M^* \neq M_{ts}(\overline{f})$. 
\end{example}

Our approach will be a little different in character, although our results will imply that the restriction of the $\overline{f}$ defined above to the set of thin dependencies whose supports are circuits {\em does} give a representation of the dual of $M_{ts}(f)$. We shall proceed by giving a self-dual characterisation of the class of tame thin sums matroids. 

\begin{lem}\label{char}
Let $M$ be a tame matroid with ground set $E$. Then $M$ is a thin sums matroid over the field $k$ iff there is for each circuit $o$ of $M$ a function $o \xrightarrow{c_o} k^*$ (here $k^*$ is the set of nonzero elements of $k$) and for each cocircuit $b$ of $M$ a function $b \xrightarrow{d_b} k^*$ such that for any circuit $o$ and cocircuit $b$ we have
\begin{equation} \label{sumeq}
\sum_{e \in o \cap b} c_o(e)d_b(e) = 0 \, .
\end{equation}
\end{lem}

\begin{proof}
Suppose first of all that we have such $c_o$ and $d_b$. Let $A$ be the set of cocircuits of $M$, and let $E \xrightarrow{f} k^A$ be defined by $f(e)(b) = d_b(e)$. We shall show that $M = M_{ts}(f)$, by showing that a set $I \subseteq E$ is $M$-dependent iff it is $M_{ts}(f)$-dependent. If $I$ is $M$-dependent, it includes some circuit $o$, and then the function extending $c_o$ to $E$ and taking the value 0 everywhere outside $o$ is a nontrivial thin dependence of $f$ which is 0 outside of $I$. If $I$ is $M_{ts}(f)$-dependent, then let $c$ be a nontrivial thin dependence of $f$ which is 0 outside of $I$, and let $s = \supp(c)$. Then for any $M$-cocircuit $b$ we have
$$\sum_{e \in E}c(e)d_b(e) = 0 \, .$$
The collection of those $e$ such that $c(e)d_b(e) \neq 0$ is $s \cap b$, which therefore can't have just one element. So by Corollary \ref{matscrawls} $s$ is a union of $M$-circuits. Since $s$ is nonempty, it is therefore $M$-dependent, and therefore so is $I$.

Conversely, let $M$ be given as $M_{ts}(f)$ for some $E \xrightarrow{f} k^A$. For each circuit $o$ of $M$, pick some thin dependence 
$\hat{c}_o$ of $f$ with support $o$, and let $c_o = \hat{c}_o \restric_o$. Now let $b$ be any cocircuit of $M$, 
and fix some $e_b \in b$.
By Lemma \ref{o_cap_b}, we can find for each $e \in b - e_b$ some circuit $o(e)$ of $M$ such that $o(e) \cap b = \{e_b, e\}$. 
We define the map $b \xrightarrow{d_b} k^*$ to be $1$ at $e_b$ and $-\frac{c_{o(e)}(e_b)}{c_{o(e)}(e)}$ for $e \in b-e_b$ 
(note that this choice ensures that \eqref{sumeq} holds for $b$ and each $o(e)$). 

Let $o$ be any circuit of 
$E$. It remains to show that $\sum_{e \in o \cap b} c_o(e)d_b(e) = 0 $.
Plugging in the values for $d_b(e)$, this means that we need to show
$$\hat{c}_o(e_b) - \sum_{e \in o \cap (b - e_b)} \frac{c_o(e)c_{o(e)}(e_b)}{c_{o(e)}(e)} =0 \, .$$
That is, we need $c(e_b)=0$, where 
$$c = \hat{c}_o - \sum_{e \in o \cap (b - e_b)}\frac{c_o(e)}{c_{o(e)}(e)}\hat{c}_{o(e)} \, .$$

As $c$ is a finite linear combination of thin dependences, it is again a thin dependence.
But for any $e \in b - e_b$, we have $c(e) = \hat{c}_o(e) -\frac{\hat{c}_o(e)}{c_{o(e)}(e)}c_{o(e)}(e) = 0$. 
If $c(e_b) \neq 0$, then by Lemma \ref{mindep}, there is a circuit $o$ such that $e_b \in o \subseteq \supp(c)$, which gives $o \cap b = \{e_b\}$, a contradiction. Thus $c(e_b)=0$, as desired.

\end{proof}

\begin{thm}\label{minorclose}
The class of tame thin sums matroids is closed under duality and under taking minors.
\end{thm}
\begin{proof}
The closure under duality follows from the fact that the characterisation given in Lemma \ref{char} is self-dual. For the closure under taking minors, let $M$ be a tame thin sums matroid with functions $c_o$, $d_b$ given as in Lemma \ref{char}, and let $N = M / C \backslash D$ be a minor of $M$. For each circuit $o$ of $N$, let $\hat{o}$ be a circuit of $M$ with $o \subseteq \hat{o} \subseteq o \cup C$ (such a circuit exists by Lemma \ref{rest_cir}), and take $c_o$ to be $c_{\hat o} \restric_o$. Similarly, for each cocircuit $b$ of $N$ let $\hat b$ be a cocircuit of $M$ with $b \subseteq \hat b \subseteq b \cup D$ and let $d_b = d_{\hat b} \restric_b$. These $c_o$ and $d_b$ satisfy the conditions of Lemma \ref{char}, so that $N$ is also a thin sums matroid over $k$.
\end{proof}

\section{Overview of the connections to graphic matroids}\label{summary}

Our results on graphic matroids have been scattered through the paper. We can now make use of Proposition \ref{finthinrep} and Theorem \ref{minorclose}, and go on a short tour of the standard matroids arising from an infinite graph $G$. We shall recall why all of them are tame thin sums matroids over any field, and a couple of them are representable over any field. Since we want our results to apply to any field, we continue to work over an arbitrary fixed field $k$.

Our starting point is the most algebraic of examples, the algebraic cycle matroid, for which we gave a thin sums representation in Proposition \ref{algthin}. Applying Theorem \ref{minorclose}, we deduce that the dual $M^*_A(G)$, the skew cuts matroid of $G$, is also a thin sums matroid. We can also apply Proposition \ref{finthinrep} to deduce that the finitarisation $M_{FC}(G) = M_A(G)_{fin}$, the finite cycle matroid $M_{FC}(G)$, whose circuits are the cycles of $G$, is representable.

Applying Theorem \ref{minorclose}, its dual, the bond matroid $M_B(G)$, whose circuits are (possibly infinite) bonds, is a thin sums matroid. So by Proposition \ref{finthinrep}, the finite bond matroid $M_{FB}(G)$ is representable. Applying Theorem \ref{minorclose} one more time, we recover the fact that the topological cycle matroid $M_{C}(G)$ is a thin sums matroid. 

We could, of course, continue this process further, but it quickly becomes periodic, as sketched out in the following diagram:

$$
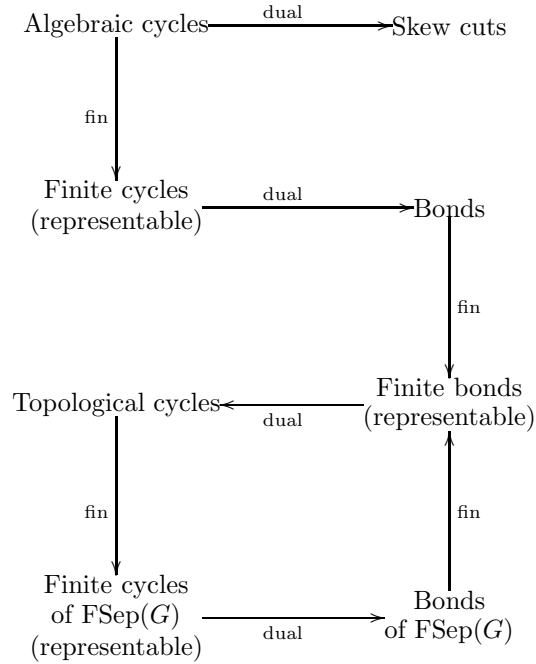
\xymatrix{*\txt{Algebraic cycles} \ar[rr]^{\dual} \ar[dd]_{\fin} && *\txt {Skew cuts} \\&&\\ *\txt{Finite cycles \\ (representable)} \ar[rr]^{\dual} && *\txt{Bonds} \ar[dd]^{\fin} \\&&\\ *\txt{Topological cycles} \ar[dd]_{\fin} && *\txt{Finite bonds \\ (representable)} \ar[ll]^{\dual} \\&&\\ *\txt{Finite cycles \\ of FSep($G$) \\ (representable)} \ar[rr]_{\dual} && *\txt{Bonds \\ of FSep($G$)} \ar[uu]_{\fin} }$$

Here FSep($G$) is the finitely separable quotient of $G$, obtained from $G$ by identifying any two vertices which cannot be separated by removing only finitely many edges from $G$. 

It would be interesting to explore the consequences of applying a similar process in other contexts, such as the simplicial setting.

\bibliographystyle{plain}
\bibliography{literatur}

\begin{thebibliography}{1}

\bibitem{union2}
E.~Aigner-Horev, J.~Carmesin, and J.~Fr{\"o}hlich.
\newblock Infinite matroid union~{II}.
\newblock Preprint (2011).

\bibitem{RD:HB:graphmatroids}
H.~Bruhn and R.~Diestel.
\newblock Infinite matroids in graphs.
\newblock arXiv:1011.4749 [math.CO], 2010.

\bibitem{matroid_axioms}
H.~Bruhn, R.~Diestel, M.~Kriesell, R.~Pendavingh, and P.~Wollan.
\newblock Axioms for infinite matroids.
\newblock arXiv:1003.3919 [math.CO], 2010.

\bibitem{basis}
H.~Bruhn and A.~Georgakopoulos.
\newblock Bases and closed spaces with infinite sums.
\newblock Preprint 2007.

\bibitem{Higgs:axioms}
D.A. Higgs.
\newblock Infinite graphs and matroids.
\newblock Recent Progress in Combinatorics, Proceedings Third Waterloo
  Conference on Combinatorics, Academic Press, 1969, pp. 245–53.

\bibitem{WILD}
N.Bowler and J.Carmesin.
\newblock Matroids with an infinite circuit-cocircuit intersection.
\newblock In preparation.

\bibitem{source_of_IE}
J.~Wojciechowski.
\newblock Infinite matroidal version of hall's matching theorem.
\newblock {\em J. London Math Soc.}, 71:563--578, 2005.

\end{thebibliography}

\end{document}